\documentclass[12pt]{article}

\usepackage{amsmath}
\usepackage{amsfonts}
\usepackage{amssymb}
\usepackage{graphicx}
\usepackage{enumerate}
\usepackage{amsthm}

\newcommand{\Nat}{{\mathbb N}}

\newcommand{\Diff}{{\mathbb D}}

\newtheorem{prop}{Proposition}
\newtheorem{theorem}{Theorem}

\newtheorem{lemma}{Lemma}

\begin{document}

\title{\bf FAMILIES OF GRAPH-DIFFERENT HAMILTON PATHS\footnote{A talk based
  on the material in this paper is accepted for presentation at the 
EuroComb'11 conference.}} 
\author{
{\bf J\'anos K\"orner}\thanks{Email: korner@di.uniroma1.it}\\
Department of Computer Science, Sapienza University\\
Rome, Italy 
\and   
{\bf Silvia Messuti}\thanks{Email: silvia.messuti@gmail.com}\\
Department of Computer Science, Sapienza University\\
Rome, Italy
\and
{\bf G\'abor Simonyi}\thanks{Email: simonyi@renyi.hu}\\
R\'enyi Institute of Mathematics\\
Budapest, Hungary} 
\date{}


\begin{titlepage}
\maketitle
\thispagestyle{empty}

\begin{abstract}
Let $\mathbb{D}\subseteq \Nat$ be an arbitrary subset of the natural numbers. For every $n$, let $M(n, \mathbb{D})$ be the maximum of the 
cardinality of a set of Hamiltonian paths in the complete graph $K_n$ such that the union of any two paths from the family contains a not necessarily induced cycle of some length from $\mathbb{D}$. We determine or bound the asymptotics of $M(n, \mathbb{D})$ in various special cases. This problem is closely related to that of the permutation capacity of graphs and constitutes a further extension of the problem area around Shannon capacity. We also discuss how to generalize our cycle-difference problems and present an example where cycles are replaced by 4-cliques. These problems are in a natural duality to those of graph intersection, initiated by Erd\H os, Simonovits and S\'os. The lack of kernel 
structure as a natural candidate for optimum makes our problems quite challenging.
\end{abstract}
\end{titlepage}


\section{Introduction}\label{intro}
Apart from scattered examples cited in \cite{SiS}, the interest in intersection problems in combinatorics begins with the seminal paper \cite{EKR}.
Graph intersection problems have been studied in a series of articles by 
Simonovits and S\'os, starting with \cite{MV}. Their famous conjecture from 1976 states that the largest family of subgraphs of $K_n$ such that 
any two of them contain a triangle in their intersection, is unique (up to isomorphism) and is the family consisting of all the subgraphs containing a fixed 
triangle. This beautiful and long--standing conjecture was just proved by Ellis, Filmus and Friedgut \cite{EFF}. The unique optimum for this problem is a 
{\em kernel structure}, consisting of all the graphs containing a ``common kernel'', the fixed triangle. 

In an intersection problem we are faced with a set on which a similarity relation is defined and we are looking for the largest cardinality of a subset any two members of which are in the given relation. We say that a relation is of {\em similarity}, if it is reflexive and for its validity between two elements it is sufficient that it be valid for some projection of the two. (We will call this {\em local verifiability}). In fact, graphs having a triangle in their intersection must contain a triangle (reflexivity) and further, if two of their subgraphs, induced by the same set of vertices (projection) contain a triangle in their intersection, then so do the two graphs. A kernel structure can be defined 
in the same manner for any relation of similarity, and it often gives the solution for the corresponding extremal problem. 
For instance, a family of subgraphs of $G$ such that any two of them contain a path of length $3$ in their intersection would have cardinality $\frac{1}{2^{3}}\cdot2^{|E(G)|}$. However, for this problem the kernel structure does not give the optimal solution, since an intersecting family of size $\frac{17}{128}\cdot2^{|E(G)|}$ is exhibited in \cite{C}.

In contrast to the above, for the extremal problems we are about to discuss here, no kernel structure can be defined. Just as with intersection problems, we 
consider a binary relation on a finite set and want to determine the largest cardinality of a subset any two elements of which are in the relation. 
However, for each of our problems, the defining relation is what we call a {\em diversity} relation. We will say that a relation is of diversity, if it is irreflexive and  locally verifiable. Such problems abound in information theory. A primary example is Shannon's graph capacity \cite{S} that was the starting point for a line of research concerning permutations \cite{KM}.


\section{Paths and cycles}

In a series of papers beginning with \cite{KM} the maximum cardinality of sets as above, for various diversity relations between  permutations defined on $[n]$ is studied. We now turn to such problems for graphs. A permutation of $[n]$ can be represented as a directed Hamiltonian path in $K_n$ in the usual manner. However, whereas the problems treated in the series of papers from \cite{KM} to \cite{sev} are closely related to graph capacity and in fact, generalize that concept, here we are dealing with relations that have their natural formulations only in purely graph--theoretic terms. Similar problems can be defined for 
families of subgraphs of $K_n$ other than paths.

Let $\Diff\subseteq\Nat$ be an arbitrary subset (finite or infinite) of the natural numbers. We will say that two undirected Hamiltonian paths in $K_n$ are $\Diff$--cycle different, if the union of the two paths contains a cycle of length from $\Diff$. We denote by $M(n, \Diff)$ the largest cardinality of a set of pairwise $\Diff$--cycle different Hamiltonian paths in $K_n$. Obviously, if $\Diff=\Nat$, every two Hamiltonian paths are $\Nat$--cycle different, and we have 
$$M(n, \Nat)=\frac{n!}{2}.$$

It is a little bit less obvious that if $\Diff=\overline{2\Nat}$, i. e., if two paths are in relation when their union contains at least one cycle of odd length, then

\begin{prop}\label{odd}
$$M(n, \overline{2\Nat})={n \choose {\lfloor \frac{n}{2} \rfloor}}$$
if n is odd, and
$$M(n, \overline{2\Nat})=\frac{1}{2}{n \choose {\frac{n}{2}}}$$
if n is even.
\end{prop}

\begin{proof}
Every Hamiltonian path uniquely belongs to an almost balanced (i.e., with the two classes that have cardinalities differing by at most 1) complete bipartite subgraph of $K_{n}$. Further, two Hamiltonian paths whose union 
contains an odd cycle, cannot belong to the same bipartite subgraph. This establishes the number of almost balanced complete bipartite graphs in $K_{n}$ as an upper bound. 
On the other hand, if two paths determine different almost balanced complete bipartite subgraphs, then their union must contain an odd cycle.
\end{proof}

Clearly, we have $M(n, {3})\leq M(n, \overline{2\Nat})$. Can equality hold? (Such coincidences do occur, e. g. in the theorem of Mantel--Tur\'an, or 
in the case of triangle--intersecting graph families \cite{EFF}.)
In the case $n=5$ this equality is proven.
\begin{prop}
$$M(5,3)=M(5,\overline{2\Nat})=10.$$
\end{prop}
\begin{proof}
The upper bound $M(5,3)\leq 10$ follows from Proposition \ref{odd}. 
As for the lower bound, observe that every Hamiltonian path is uniquely contained as a subgraph both in a $C_{5}$ and also in a complete bipartite graph $K_{2,3}$. We define a bipartite graph whose vertices correspond to the $C_{5}$ and $K_{2,3}$ subgraphs of $K_{5}$ so that two vertices are adjacent if one is a $C_{5}$, the other is a $K_{2,3}$ and the two graphs intersect in a Hamiltonian path. Note that a family of triangle--different Hamiltonian paths is characterized by not containing two paths from the same $C_{5}$ nor from the same $K_{2,3}$. Hence, we are done if we show that our bipartite graph contains 10 pairwise vertex-disjoint edges. But this is immediate by Hall's theorem, since our graph is biregular.
\end{proof}

Somewhat surprisingly, the situation is quite different if we require even cycles. It is immediately obvious that super--exponentially large families of pairwise even--cycle--different paths exist, since all the Hamiltonian paths contained in the same almost balanced complete bipartite graph form such a family. 
We will show that these families are still far from being optimal.

\begin{lemma}
A graph containing two cycles which share a unique path necessarily contains an even cycle.
\end{lemma}

\begin{proof}
If one of the two cycles has even length, then the fact is obviously true, so suppose they both have odd length. It is easy to see that a third cycle lies in the graph and it is obtained by omitting from the graph the edges and the inner vertices of the shared path. Hence, if the two cycles have length $l_{1}$ and $l_{2}$, respectively, and the shared path contains $s$ edges, then the third cycle has length
$$l_{1}+l_{2}-2s\equiv 0\,(\textrm{mod}\,2).$$
\end{proof}


\begin{theorem}\label{gre}
$$\frac{n!}{2[n-3+{\lceil\frac{n}{2}\rceil+1 \choose 2 }]} \leq M(n, 2\Nat)\leq \frac{n!}{2n}$$
if $n$ is odd, and
$$\frac{n!}{2{\frac{n}{2}+1 \choose 2 }} \leq M(n, 2\Nat)\leq \frac{n!}{8}$$
if $n$ is even.
\end{theorem}

\begin{proof}
Let $\mathcal{H}_{n}$ be the set of all the Hamiltonian paths in $K_{n}$. The following greedy procedure constructs a family $\mathcal{F}$ of $2\Nat$--cycle different paths:
At each round, select an arbitrary $H$ from the remaining part of $\mathcal{H}_{n}$ and let $\overline{H}$ be the set of all the paths $H'$ in $\mathcal{H}_{n}$, including $H$ itself, such that $H\cup H'$ does not contain even cycles. Then, add $H$ to $\mathcal{F}$ and delete all the paths in $\overline{H}$ from $\mathcal{H}_{n}$. 

The procedure is correct. Indeed, let $H_{i}$ and $H_{j}$ be the paths added to $\mathcal{F}$ at round $i$ and $j$ respectively, with $i<j$, and suppose that $H_{i}\cup H_{j}$ does not contain even cycles. This means that $H_{j}$ would be included in $\overline{H_{i}}$ at round $i$, hence it would be deleted from $\mathcal{H}_{n}$ at the end of round $i$, thus it could not be selected at round $j$.

To prove that the constructed family has the stated cardinality, we show that for arbitrary $H \in \mathcal{H}_{n}$ 
$$|\overline{H}|= n-3+{\lceil\frac{n}{2}\rceil+1 \choose 2 }$$ if $n$ is odd, and
$$|\overline{H}|= {\frac{n}{2}+1 \choose 2}$$ if $n$ is even.

For the sake of simplicity let $H=(1,2,\dots,n)$. We want to know which edges in $K_{n}$ are allowed to belong to a path $H'$ such that $H\cup H'$ does not contain even cycles. Surely $H'$ has an edge $\{x,y\}$ such that $x+1<y$. This defines a cycle $C=(x,x+1,\dots,y,x)$ in $H\cup H'$ and since $C$ must have odd length, we have $y-x\equiv 0\,(\textrm{mod}\,2)$. 

We partition the remaining vertices as follows:
\begin{itemize}
\item $head=\{1,2,\dots,x-1\}$
\item $cycle=\{x+1,\dots,y-1\}$
\item $tail=\{y+1,\dots,n\}$
\end{itemize}
and we consider an arbitrary edge $\{i,j\}$ in $H'$. Suppose $i+1<j$, then $(i,i+1,\dots,j,i)$ is a cycle in $H\cup H'$ which shares a path with $C$ in the following cases (thus producing an even cycle):
\begin{enumerate}[i.]
\item\label{ihjc} if $i\in head$ and $j\in cycle$ then the shared path is $(x,\dots,j)$  
\item\label{icjt} if $i\in cycle$ and $j\in tail$ then the shared path is $(i,\dots,y)$
\item\label{ihjt} if $i\in head$ and $j\in tail$ then the shared path is $(x,\dots,y)$
\item\label{icjc} if $i\in cycle$ and $j\in cycle$ then the shared path is $(i,\dots,j)$
\end{enumerate}

From (\ref{ihjc}), (\ref{icjt}) and (\ref{icjc}) it follows that a vertex in $cycle$ can be incident only to the edges in $H$, but since $\{x,y\}\in H'$ there must be exactly one vertex  $i^{*}\in\{x\}\cup cycle$ such that $\{i^{*},i^{*}+1\}\notin H'$. 

We have two cases:
\begin{enumerate}[1.]
\item If $y-x<n-1$ then $i^{*}$ must be one among $x$ and $y-1$ since else the cycle would be disconnected from the rest of the path. Suppose $\{x,x+1\}\in H'$ (the case for $\{y-1,y\}$ is symmetric). Then $y-1$ is an endpoint of $H'$ (from (\ref{ihjc}) and (\ref{icjt})), $x=1$ and $head=\emptyset$ since else $head$ would be disconnected from the rest of the path (from (\ref{ihjc}) and (\ref{ihjt})), and
$$H'=(y-1,y-2,\dots,1,y,tail).$$

The above considerations can be recursively applied to $tail$. Suppose there is an edge $\{x',y'\}$ with $x'+1<y'$ and $y'-x'\equiv 0\,(\textrm{mod}\,2)$, hence an odd cycle $C'=(x',x'+1,\dots,y',x')$ in $tail$. 
Then we get a new partition
\begin{itemize}
\item $head=V(C)\cup\{y+1,\dots,x'-1\}$
\item $cycle=\{x'+1,\dots,y'-1\}$
\item $tail=\{y'+1,\dots,n\}$
\end{itemize}
However, since $head$ contains the vertices in $C$, it cannot be empty, thus the other endpoint of $H'$ is $x'+1$. Furthermore, $y'=n$ and we get
$$H'=(y-1,y-2,\dots,1,y,\dots,x',n,n-1,\dots,x'+1).$$

To calculate the cardinality of $\overline{H}$, we have to determine which values for $y$ and $x'$ in $H'$ result in having only odd cycles in $H\cup H'$. Since $x=1$ and $y-x\equiv 0\,(\textrm{mod}\,2)$, we have $\lceil\frac{n}{2}\rceil$ possible values for $y$:
$$y\in\{1, 3, \dots,n\}\quad\textrm{if $n$ is odd}$$
$$y\in\{1, 3, \dots,n-1\}\quad\textrm{if $n$ is even}$$
Similarly, we have $y'=n$ and $y'-x'\equiv 0\,(\textrm{mod}\,2)$. Moreover, it must be $y\leq x'$, thus we have that the possible choices for $x'$ depend on the value of $y$
$$x'\in\{n,n-2\dots,y\}\quad\textrm{if $n$ is odd}$$
$$x'\in\{n,n-2\dots,1+y\}\quad\textrm{if $n$ is even}.$$
This gives 
$$|\overline{H}| \geq {\lceil\frac{n}{2}\rceil+1 \choose 2}. $$

\item If $y-x=n-1$ then $n$ is odd, $head=tail=\emptyset$ and $C$ is a Hamiltonian cycle, thus $i^{*}$ can be an arbitrary vertex in $\{x\}\cup cycle$. This gives additional $n-3$ paths in $\overline{H}$ thus we get 
$$|\overline{H}| = n-3+{\lceil\frac{n}{2}\rceil+1 \choose 2}$$ if $n$ is odd, and
$$|\overline{H}| = {\lceil\frac{n}{2}\rceil+1 \choose 2}$$ if $n$ is even.

\end{enumerate}

As for the upper bound, every Hamiltonian path uniquely belongs to a Hamiltonian cycle and two Hamiltonian paths whose union contains an even cycle cannot belong to the same Hamiltonian cycle if $n$ is odd. Since a Hamiltonian cycle contains $n$ Hamiltonian paths, we have $\frac{n!}{2n}$ as an upper bound.

If $n$ is even, for each Hamiltonian path $H$, one can construct a set $S(H)$ of Hamiltonian paths whose pairwise union contains only odd cycles in the following manner. For the sake of simplicity set $H=(1,2,\dots,n)$ once again. We define $S(H)=\{H,H_{l},H_{r},H_{lr}\}$ such that
$$H_{l}=(2,1,3,\dots,n-2,n-1,n)$$
$$H_{r}=(1,2,3,\dots,n-2,n,n-1)$$
$$H_{rl}=(2,1,3,\dots,n-2,n,n-1).$$

Note that this set cannot be enlarged. Indeed, consider two paths 
$$H_{i}=(2i,2i-1,\dots,1,2i+1,\dots,n)$$
and
$$H_{j}=(2j,2j-1,\dots,1,2j+1,\dots,n)$$
with $i<j$. Then $H_{i}\cup H_{j}$ contains an even cycle $(1,2i+1,\dots2j+1)$. The same applies if the cycle is located at the other endpoint.

It is easy to see that for arbitrary Hamiltonian paths $A$ and $B$ we have 
$$|\{H'\in \mathcal{H}_{n}:A\in S(H')\}|=|\{H'\in \mathcal{H}_{n}:B\in S(H')\}|$$
and that only one path in each set $S(H)$ can belong to an optimal family, hence we get
$$M(n, 2\Nat)\leq\frac{1}{4}\cdot\frac{n!}{2}=\frac{n!}{8}.$$

\end{proof}


It is interesting to investigate our problem for different cycle lengths. In particular, starting from the observation that the number of $2\Nat$--cycle different paths is superexponential while that of $\overline{2\Nat}$--cycle different paths is only exponential, we wondered whether the same would apply if we considered the Hamiltonian paths in $K_{n}$ the union of any two of which has a cycle of length divisible, respectively non-divisible, by $c>2$. Unlike for $c=2$, we have a superexponential bound for both cases.

\begin{theorem}\label{more}
For arbitrary $c \in \Nat$ there exists a $d \leq 1$ such that 
$$(dn)!\leq M(n, c\Nat).$$
\end{theorem}

\begin{proof}
We describe a base construction which provides a weak lower bound for $c=2$ but is easily generalized for every $c>2$.

Let $M$ be a fixed perfect matching in $K_{n}$ (suppose $n$ is even for the sake of simplicity) and $(A,B)$ a partition of the vertices such that every edge $e=\{a_{e},b_{e}\}$ in $M$ has an endpoint in $A$ and an endpoint in $B$. A Hamiltonian path in $K_{n}$ can be constructed by choosing a permutation of the edges in $M$ and joining the $B$-vertex of each edge with the $A$-vertex of the subsequent. We say that an edge is a $matching$ edge if it belongs to $M$, while it is a $linking$ edge if it joins two $matching$ edges. 

It is easy to see that two arbitrary distinct Hamiltonian paths resulting from this construction have only even cycles in their union since this is a bipartite graph. This gives
$$(\frac{1}{2}n)!\leq M(n, 2\Nat).$$
To extend this construction for every $c>2$ we need to show that the union graph contains at least one even cycle with exactly the same number of $matching$ edges and $linking$ edges. 

Let $H$ and $H'$ be two paths constructed in the above manner. We label the $matching$ edges $1,2,\dots,n/2$ as they appear in $H$, thus we have 
$$H=a_{1},b_{1},a_{2},b_{2},\dots,b_{n/2-1},a_{n/2},b_{n/2}.$$
Since $H'$ results from a different permutation, there must be at least one $matching$ edge $x$ such that $x+1$ precedes $x$ in $H'$. Then, the subpath between $x+1$ and $x$ in $H'$, together with the $linking$ edge $\{b_{x},a_{x+1}\}$ in $H$, gives the cycle 
$$C=(a_{x+1},b_{x+1},\dots,a_{x},b_{x},a_{x+1})$$ 
in $H\cup H'$. Since the subpath begins and ends with $matching$ edges, we have that their number in $C$ is exactly the same as that of the $linking$ edges. 

That being stated, we can get a family of $c\Nat$--cycle different Hamiltonian paths by partitioning the set of the vertices in $n/c$ pairwise disjoint paths to replace the $matching$ edges in the above construction (suppose $c|n$ for the sake of simplicity). This can be pictured like inserting $c-2$ vertices on each $matching$ edge, thus a cycle with $k$ $matching$ edges and $k$ $linking$ edges enlarges its length to
$$k+k+(c-2)k=ck\equiv 0\,(\textrm{mod}\,c)$$
and we get
$$(\frac{1}{c}n)!\leq M(n, c\Nat)$$
for every value of $c$.
\end{proof}

\begin{theorem}\label{notmore}
For arbitrary integer $c>2$ there exists a $d^*  \leq1$ such that 
$$(d^*n)!\leq M(n, \overline{c\Nat}).$$
\end{theorem}

\begin{proof}
The above construction, with a slight variation, provides a family of Hamiltonian paths whose pairwise union contains at least one cycle of length $k\equiv 2\,(\textrm{mod}\,c)$.
Instead of considering all the permutations of the $matching$ edges, we fix the first one and vary the position of the other $n/2-1$, thus obtaining $(n/2-1)!$ paths.
Let $H$ and $H'$ be two such paths and label the $matching$ edges $1,2,\dots,n/2$ as they appear in $H$. 
Since $H'$ comes from a different permutation we have
$$H=a_{1},b_{1},\dots,a_{x},b_{x},a_{x+1},b_{x+1},\dots,a_{n/2},b_{n/2}$$
and
$$H'=a_{1},b_{1},\dots,a_{x},b_{x},a_{y},b_{y},\dots,a_{x+1},b_{x+1},\dots$$
with $y\neq x+1$. Thus $H\cup H'$ contains the cycle 
$$C=(b_{x},a_{y},b_{y},\dots,a_{x+1},b_{x}).$$
Note that $\{b_{x},a_{y}\}$ and $\{b_{x},a_{x+1}\}$ are both $linking$ edges, thus if the number of $matching$ edges in $C$ is $k$, that of the $linking$ edges is $k+2$. 

When we replace the $matching$ edges with pairwise vertex-disjoint paths of length $c$ as seen in the previous proof we get a cycle of length
$$(k+2)+k+(c-2)k=ck+2\equiv 2\,(\textrm{mod}\,c)$$
and a family with
$$(\frac{1}{c}n-1)!\leq M(n, c\Nat)$$
$\overline{c\Nat}$--cycle different paths for every value of $c$.
\end{proof}

A simpler construction provides a better lower bound for almost every value of $c$.

\begin{prop}
For arbitrary integer $c>1$ other than $2$ and $4$    
$$\frac{n!}{2{n \choose 2}}\leq M(n, \overline{c\Nat}).$$
\end{prop}

\begin{proof}
Let us fix two vertices in $K_{n}$. The union of every two Hamiltonian paths having those vertices as endpoints contains at least two cycles which share a common path. If both of these cycles have length from ${c\Nat}$, we need to show that they yield a cycle whose length must be from $\overline{c\Nat}$.

Consider two Hamiltonian paths in $K_{n}$ having the same endpoints: the $blue$ path and the $red$ path. In their union graph $G$, we say an edge is $purple$ if it belongs to both paths, while it is $blue$ or $red$ if it belongs to that path only. Moreover, we say a vertex is $purple$ if it is incident to $purple$ edges only and we label the vertices $1,2,\dots,n$ as they appear in the $blue$ path.

Let $x$ be the smallest non-$purple$ vertex in $G$ and $\{x, y\}$ the $red$ edge incident to $x$. This defines a cycle in $G$, namely $(x,x+1,\dots,y-1,y,x)$. 
Also, $x+1$ is not a $purple$ vertex, because $\{x,x+1\}$ is a $blue$ edge since else $x$ would have $red$ degree $3$, thus there is at least one $red$ edge incident to $x+1$, say $\{x+1,z\}$, which defines a cycle $(x+1,x+2,\dots,z-1,z,x+1)$.

Suppose every cycle in $G$ has length from $c\Nat$. Thus 
$$y-x+1\equiv 0\,(\textrm{mod}\,c)$$
and
$$z-(x+1)+1\equiv 0\,(\textrm{mod}\,c)$$
must hold.

These two cycles share a subpath of the $blue$ path which contains $s$ edges and define a third cycle $C=(x,x+1,z,\dots,y,x)$ whose length is given by 
$$(y-x+1)+(z-(x+1)+1)-2s.$$

We will show that $C$ cannot have a length from $c\Nat$. Two cases are to be considered:
\begin{enumerate}[1.] 
\item
If $z\geq y$, then we have $s=y-(x+1)$ and $C$ has length 
$$(y-x+1)+(z-(x+1)+1)-(2y-2x-2)$$ 
which cannot be congruent to $0$ modulo $c$ since else we would get
$$y-x+1\equiv 2(y-x+1)-4\,(\textrm{mod}\,c)\quad \Rightarrow \quad 0\equiv -4\,(\textrm{mod}\,c).$$

\item
If $z<y$, then we have $s=z-(x+1)$, thus $C$ has length
$$(y-x+1)+(z-(x+1)+1)-(2z-2x-2)$$ 
which, again, cannot be congruent to $0$ modulo $c$ since else we would get
$$z-(x+1)+1\equiv 2[z-(x+1)+1]-2\,(\textrm{mod}\,c)\quad \Rightarrow \quad 0\equiv -2\,(\textrm{mod}\,c).$$
\end{enumerate}
\end{proof}


\section{Related problems}

The problem about cycles in the union of paths can be generalized in the following manner. Let $\mathcal{F}$ be a family of finite graphs and $G$ a finite graph such that $G$ has no 
subgraph (even not induced) isomorphic to any $F\in\mathcal{F}$. What is the maximum number of copies of $G$ among the subgraphs of $K_n$ such that the union of any two contains a copy of a graph from $\mathcal{F}$? 

For instance, one can consider the following problem: how many Hamiltonian paths can we have in $K_{n}$ if the union of any two of them must contain a $K_{4}$? 

\begin{prop}
Let $M(n, K_{4})$ be the largest cardinality of a set of Hamiltonian paths in $K_{n}$ such that the union of any two of them contains a $4$--clique. Then we have
$$2^{\frac{n}{4}}\leq M(n, K_{4})\leq (n+1)^{2}\Big(\frac{3}{2}\Big)^{n-1}.$$
\end{prop}

\begin{proof}
It is easy to see that for $n=4$ there are exactly two Hamiltonian paths which form a $K_{4}$, for instance $(1,2,3,4)$ and $(2,4,1,3)$. For $n>4$ we can partition the vertices into $4$--tuples and link together the subpaths, thus obtaining 
$$2^{\frac{n}{4}}\leq M(n, K_{4}).$$

For the upper bound, on the one hand we have that every Hamiltonian path is a subgraph of a unique almost balanced complete bipartite graph. On the other hand, the union of two Hamiltonian paths containing a $K_{4}$ is at least $4$--chromatic, so each almost balanced complete bipartite graph can contain at most one path of the optimal family. This gives
$$M(n, K_{4})\leq \frac{1}{2}{n \choose {\frac{n}{2}}}.$$

The upper bound can be improved by considering almost balanced complete tripartite graphs. Every Hamiltonian path is a subgraph of the same number of almost balanced complete tripartite graphs, but it still holds that every such graph contains at most one path of the optimal family, since its chromatic number is $3$. 
Hence, we need to lower bound the number of Hamiltonian paths contained in an almost balanced complete tripartite graph. 

For the sake of simplicity, consider the case of $3|n$. We fix a permutation of the $\frac{n}{3}$ elements in each class of the partition independently in $[(\frac{n}{3})!]^{3}$ many ways. For every such ordering we have to specify in which order we pass from one class to the other, and this can be represented by a string on the alphabet $\{a,b,c\}$ (where $a,b,c$ are the indices of the three classes) in which two adjacent symbols are distinct. These strings are $3\cdot2^{n-1}$, as for each coordinate after the first we have two choices. However, we are dealing with a balanced tripartite graph and each vertex must appear exactly once in the path, thus we have to consider only those strings in which an equal number of $a$, $b$ and $c$ occurs.

We define the $type$ of a string by the number of occurrences for each symbol. It is easy to see that there are at most $(n+1)^{2}$ types for a ternary string of length $n$, and that the type which contains the highest number of strings is, by symmetry, that with $\frac{n}{3}$ occurrences for each symbol. Hence, the number of paths in a balanced complete tripartite graph is lower bounded by
$$\Big[\Big(\frac{n}{3}\Big)!\Big]^{3}\cdot\frac{3\cdot2^{n-1}}{(n+1)^{2}}.$$
Therefore, we have
$$M(n, K_{4})\leq \frac{n!}{3[(\frac{n}{3})!]^{3}\cdot\frac{2^{n-1}}{(n+1)^{2}}}\leq(n+1)^{2}\Big(\frac{3}{2}\Big)^{n-1}$$
where the last inequality holds since
$$\frac{n!}{[(\frac{n}{3})!]^{3}}\leq 3^{n}.$$
\end{proof}

One can consider many further problems along the same lines. 



\begin{thebibliography}{10}\label{bibliography}

\bibitem{sev} Brightwell, G., G. Cohen, E. Fachini, M. Fairthorne,
  J. K\"orner, G. Simonyi, and \'A. T\'oth, \emph{Permutation capacities of
    families of oriented infinite paths}, SIAM J. Discrete Math. \textbf{24}
  (2010), 441--456,
  
\bibitem{C} Christofides, D. \emph{A counterexample to a conjecture of Simonovits and S\'os}, Manuscript  

\bibitem{EFF} Ellis, D., Y. Filmus and E. Friedgut, \emph{Triangle--intersecting families of graphs}, European J. Math., to appear,

\bibitem{EKR} Erd\H os, P., Chao Ko and R. Rado,  \emph{Intersection theorems for systems of finite sets}, 
Quart. J. Math. Oxford Ser. 2, 12(1961), pp. 313--320, 

\bibitem{KM} K\"orner, J. and  C. Malvenuto, \emph{Pairwise colliding
    permutations and the capacity of infinite graphs},  \textbf{20} (2006),
  203--212,

\bibitem{S} Shannon, C. E.,  \emph{The zero-error capacity of a noisy
    channel}, IEEE Trans. Inform. Theory \textbf{2} (1956), 8--19, 
 
\bibitem{SiS} Simonovits, M. and V. T. S\'os, \emph{Intersection theorems on structures}, 
Annals of Discrete Mathematics, 6(1980), pp. 301--313,

\bibitem{MV} Simonovits, M. and V. T. S\'os, \emph{Intersection theorems for graphs}, 
Colloques Intern. CNRS no. 260, Probl\`emes combinatoires et th\'eorie des graphes.
pp. 389--391,

\bibitem{VLW} van Lint, J. H. and R. M. Wilson, \emph{A course in Combinatorics, 2nd edition}, Cambridge University Press (2001).

\end{thebibliography}
\end{document}